\documentclass[11pt]{article}  
\usepackage{geometry}                
\usepackage{amsmath,amsthm,amscd,amssymb,mathabx}

\usepackage{color, comment}
\usepackage{graphicx}
\usepackage{xypic}
\usepackage{xcolor,colortbl}
\usepackage{pinlabel}

\definecolor{Gray}{gray}{0.85}
\newtheorem{theorem}{Theorem}

\newtheorem{conjecture}[theorem]{Conjecture}
\newtheorem{corollary}[theorem]{Corollary}
\newtheorem{lemma}[theorem]{Lemma}
\newtheorem{proposition}[theorem]{Proposition}

\makeatletter
\newtheorem*{rep@theorem}{\rep@title}
\newcommand{\newreptheorem}[2]{%
\newenvironment{rep#1}[1]{%
 \def\rep@title{#2 \ref{##1}}%
 \begin{rep@theorem}}%
 {\end{rep@theorem}}}
\makeatother

\newreptheorem{theorem}{Theorem}
\newreptheorem{lemma}{Lemma}

\theoremstyle{definition}
\newtheorem{definition}[theorem]{Definition}
\newtheorem{example}[theorem]{Example}

\newtheorem{remark}[theorem]{Remark}

\newcommand{\Z}{\mathbb{Z}}
\newcommand{\Q}{\mathbb{Q}}

\newcommand{\fix}{\operatorname{Fix}}

\newcommand{\hf}{\widehat{HF}}
\newcommand{\rkh}{\widebar{Kh}}

\newcommand{\ord}{\operatorname{ord}}
\newcommand{\isom}{\cong}
\newcommand{\gam}{\tilde{\gamma}}
\newcommand{\tGam}{\widetilde{\Gamma}}
\newcommand{\Gam}{\Gamma}

\newcommand{\bdc}{\Sigma(K)}
\newcommand{\os}{Ozsv\'ath and Szab\'o }

\title{Cosmetic surgery in L-spaces and nugatory crossings}
\date{}
\author{Tye Lidman and Allison H.\ Moore }

\begin{document}

\maketitle

\begin{abstract}
The cosmetic crossing conjecture (also known as the ``nugatory crossing conjecture") asserts that the only crossing changes that preserve the oriented isotopy class of a knot in the 3-sphere are nugatory.  We use the Dehn surgery characterization of the unknot to prove this conjecture for knots in integer homology spheres whose branched double covers are L-spaces satisfying a homological condition. This includes as a special case all alternating and quasi-alternating knots with square-free determinant. As an application, we prove the cosmetic crossing conjecture holds for all knots with at most nine crossings and provide new examples of knots, including pretzel knots, non-arborescent knots and symmetric unions for which the conjecture holds.  
\end{abstract}



\section{Introduction}
Let $K$ be an oriented knot in $S^3$, and let $c$ refer to an oriented crossing in a diagram of the knot. A fundamental question is whether a crossing change at $c$ preserves the isotopy type of the knot. Let a crossing disk $D$ be an embedded disk in $S^3$ intersecting $K$ twice with zero algebraic intersection number. If $\partial D$ also bounds an embedded disk in the complement of $K$, then the crossing is called \emph{nugatory} and changing this crossing does not change the isotopy type of $K$.  A non-nugatory crossing change which preserves the oriented isotopy type of the knot is called \emph{cosmetic} and it is conjectured that no such crossing exists for a knot in $S^3$. 

\begin{conjecture}[Cosmetic crossing conjecture]\label{conj:ccc}
If $K$ admits a crossing change at a crossing $c$ which preserves the oriented isotopy class of the knot, then $c$ is nugatory.
\end{conjecture}
Remarkably, this basic question is unanswered for most classes of knots. The conjecture is attributed to X. S. Lin (see \cite[Problem 1.58]{Kirby:Problems}), and in the literature it sometimes appears as the ``nugatory crossing conjecture''. Scharlemann and Thompson proved that the unknot admits no cosmetic crossing changes \cite[Theorem 1.4]{ST}, and Torisu and Kalfagianni established the same for 2-bridge knots and fibered knots, respectively \cite{Torisu, Kalfagianni}. There exist several additional obstructions amongst genus one knots and satellites \cite{BFKP, BK}.

In this note, we prove Conjecture~\ref{conj:ccc} for knots whose branched double covers are L-spaces that satisfy a certain homological condition. Recall that an L-space $Y$ is a rational homology sphere with $\operatorname{rank} \widehat{HF}(Y) = |H_1(Y;\mathbb{Z})|$, where $\widehat{HF}$ denotes the hat flavor of Heegaard Floer homology. 

\begin{theorem}
\label{main}
Let $K$ be a knot in $S^3$ whose branched double cover $\Sigma(K)$ is an L-space. If each summand of the first homology of $\Sigma(K)$ has square-free order, then $K$ satisfies the cosmetic crossing conjecture.
\end{theorem}

Theorem~\ref{main} will be deduced from the Dehn surgery characterization of the unknot in L-spaces \cite{Gainullin, KMOS} (see Theorem \ref{unknot}). It is interesting to juxtapose Theorem \ref{main} with \cite[Theorem 1.1]{BFKP}, which implies that if a genus one knot $K$ admits a cosmetic crossing change, then $H_1(\Sigma(K))$ is cyclic of order $d^2$, for some $d\in\Z$.
 
An abundant source of knots that meet the conditions of Theorem~\ref{main} are the \emph{Khovanov thin} knots. These knots derive their definition from reduced Khovanov homology, which associates to an oriented link $L$ in $S^3$ a bigraded vector space $\rkh^{i,j}(L)$ over $\Z/2\Z$ \cite{Khovanov}. The $\rkh$--thin links are those with their homology supported in a single diagonal $\delta=j-i$ of the bigradings.  In this case, the dimension of $\rkh$ is given by the determinant of the link. By work of Manolescu and Ozsv\'ath \cite{ManOzs}, all quasi-alternating links are $\rkh$--thin, and this class includes all non-split alternating links \cite{OzSz:BDCs}. Of relevance here is the fact that the branched double cover of a $\rkh$--thin link is an L-space, which follows from the spectral sequence from $\rkh(L)$ to $\hf(-\Sigma(L))$ \cite{OzSz:BDCs} and the symmetry of Heegaard Floer homology under orientation reversal \cite{OzSz:Properties}.
 
Because the determinant of a knot is equal to the order of the first homology of its branched double cover we immediately obtain the following corollary.
\begin{corollary}
\label{squarefreedet}
A $\rkh$--thin knot with square-free determinant satisfies the cosmetic crossing conjecture.
\end{corollary}

We apply Theorem~\ref{main} (and in particular, Corollary \ref{squarefreedet}) together with previously known obstructions for two-bridge, fibered and genus one knots to affirm the cosmetic crossing conjecture for all but ten knots in the Rolfsen tables \cite{Rolfsen} of knots of ten or fewer crossings. 
\begin{theorem}
\label{classification}
Let $K$ be a knot of at most ten crossings not contained in the list
\begin{equation}\label{exceptions}
	10_{65}, 10_{66}, 10_{67}, 10_{77}, 10_{87}, 10_{98}, 10_{108}, 10_{129}, 10_{147}, 10_{164}.
\end{equation}
Then $K$ admits no cosmetic crossing changes. In particular, all knots of at most nine crossings satisfy the cosmetic crossing conjecture. 
\end{theorem}

\begin{remark}
In fact, what we prove in Theorem \ref{main} holds in a more general setting. Indeed, for any knot in an integer homology sphere whose branched double cover is an L-space, if each summand of the first homology of the branched double cover is square-free then that knot does not admit cosmetic crossing changes. This statement requires an appropriate generalization of the definition of a cosmetic crossing change to an arbitrary homology sphere. (See Remark~\ref{generalization} and Theorem~\ref{generaltheorem}.)
\end{remark}

Another interesting class of knots are those which admit a positive Dehn surgery to an L-space; such a knot is known as an \emph{L-space knot}. Since these knots are known to be fibered by work of Ni \cite{Ni}, it follows from Kalfagianni \cite{Kalfagianni} that L-space knots satisfy the cosmetic crossing conjecture. To further illustrate the techniques of the proof of the main theorem, we provide an alternate proof of this fact.

\begin{theorem}
\label{lspaceknots}
Let $K$ be an L-space knot. Then $K$ satisfies the cosmetic crossing conjecture.
\end{theorem}

In Section \ref{sec:homology}, we establish necessary background information and prove a key homological result. In Section \ref{sec:mainproofs} we prove the main result Theorem \ref{main} and its generalization, as well as Theorem \ref{lspaceknots}. In Section \ref{sec:examples}, we prove Theorem \ref{classification} and we provide new examples of pretzel knots, non-arborescent knots and symmetric unions for which the cosmetic crossing conjecture holds.


\section{Homological obstructions}
\label{sec:homology}

In order to prove Theorem~\ref{main}, we study the effects of a cosmetic crossing change in terms of the homology of the branched double cover.  The current section is devoted to understanding this, culminating in Theorem~\ref{squarefreefactors}.  

\subsection{Cosmetic and nugatory crossings}
Let $K$ be an oriented knot in $S^3$, and let $c$ denote a crossing. We will abuse notation by allowing $K$ to denote both the oriented knot and its diagram.  We will write $K^+$ and $K^-$ for diagrams identical to $K$ except possibly at the crossing $c$, which is positive at $K^+$ and negative at $K^-$. Without loss of generality, we assume the crossing $c$ is positive. 
\begin{figure}
	\centering
	\includegraphics[width=8cm]{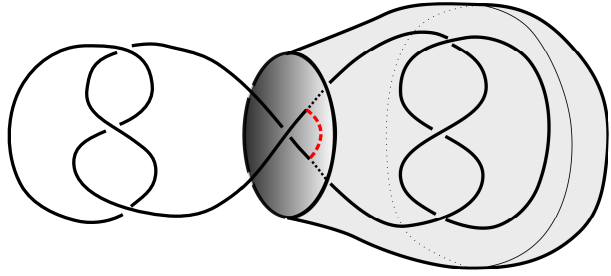}
	\caption{The crossing disk of a nugatory crossing.}
	\label{crossingdisk}
\end{figure} 
A \emph{crossing disk} for $K$ is an embedded disk $D$ that intersects $K$ transversely in its interior twice with zero algebraic intersection number, as in Figure \ref{crossingdisk}. The boundary of the crossing disk $\partial D$ is an unknot called the \emph{crossing circle}. 
If $\partial D$ bounds an embedded disk in the complement of $K$, then $c$ is called \emph{nugatory}. A \emph{crossing arc} $\gamma$ for the knot $K$ is an unknotted arc with its boundary on $K$ that may be isotoped to lie embedded in the crossing disk, again as in Figure \ref{crossingdisk}. We write $\Sigma(K)$ for the double cover of $S^3$ branched over the knot $K$. If $K$ admits a crossing change at $c$ which preserves the isotopy type of $K$, then the branched double covers $\Sigma(K^+)$ and $\Sigma(K^-)$ are orientation preserving homeomorphic, and the arc $\gamma$ lifts to a knot $\gam$ in $\Sigma(K^+)$. The complement of a neighborhood of $\gam$ in $\Sigma(K)$, is a compact, connected, oriented 3-manifold with torus boundary which we denote by $M$.  

There is a well-known correspondence between Dehn fillings of $M$ and rational tangle replacements of $K$ in $S^3$, commonly called ``the Montesinos trick." (See, for example, \cite{Gordon:Dehn} for a standard reference.)  We only state a special case here.  Take a small 3-ball $B$ containing the crossing $c$ in $K^+$, so that the sphere $\partial B$ intersects $K$ transversely in four points. Then the double cover of $S^3 - \operatorname{int}(B)$ branched over $T=K - \operatorname{int}(B\cap K)$ is the manifold $M = \Sigma(K^+)-N(\gam)$. Note that $\Sigma(K^-)$ is also obtained by a Dehn filling of $M$.  Finally, recall that the distance between any two slopes $\eta$ and $\xi$ on $\partial M$ refers to their minimal geometric intersection number and is denoted $\Delta(\eta, \xi)$.

\begin{lemma}[Montesinos trick] \label{montesinostrick} Let $\alpha$ and $\beta$ be the two slopes on $\partial M$ such that $M(\alpha) = \Sigma(K^+)$ and $M(\beta) = \Sigma(K^-)$.  Then $\Delta(\alpha, \beta) = 2$.
\end{lemma}

\subsection{Rational longitude}
\label{rationallongitude}
Now let $M$ refer to any compact, connected, oriented 3-manifold with torus boundary and $H_1(M;\Q)\isom \Q$. For homology with $\mathbb{Z}$-coefficients, we will simply omit the coefficient group. In particular, the complement of the knot $\gam$ in the rational homology sphere $\bdc$ is such a manifold. In this section we recall some facts about the rational longitude of $M$ which we will use to study the homology of Dehn fillings of $M$.  For more detail on the rational longitude, we refer the reader to Section 3.1 of Watson \cite{Watson}.  We adopt the same notation as in \cite{Watson}, and the content of our own Section \ref{rationallongitude} is paraphrased from this, included because it is necessary for the arguments which follow in Section \ref{lift}.

Let us consider the long exact sequence of the pair $(M, \partial M)$,
\[
\cdots \rightarrow H_2(M) \rightarrow H_2(M,\partial M)\rightarrow H_1(\partial M)\stackrel{i_*}{\rightarrow} H_1(M) \rightarrow H_1(M, \partial M) \rightarrow \cdots.
\]
Via exactness and Poincar\'e-Lefschetz duality, it follows that the rank of $i_*$ is one. Therefore with $\Z$--coefficients, $i_*$ maps one of the $\Z$ summands of $H_1(\partial M)\isom \Z\oplus\Z$ injectively to $H_1(M)\isom \Z\oplus H$ (where $H$ is some finite abelian group). Additionally, $\ker(i_*)$ is generated by $k\lambda_M$ for some primitive homology class $\lambda_M\in H_1(\partial M)$ and some nonnegative integer $k$. The homology class $\lambda_M$ is uniquely defined up to sign, which determines a well-defined slope in $\partial M$, giving rise to the following definition.
\begin{definition}
The \emph{rational longitude} $\lambda_M$ is the unique slope in $\partial M$ such that $i_*(\lambda_M)$ is finite order in $H_1(M)$.
\end{definition}
We begin with the observation of \cite{Watson} that the order of the first homology of the manifold $M(\eta)$ obtained by Dehn filling $M$ along some slope $\eta$ is determined by $\Delta(\eta, \lambda_M)$.
\begin{lemma}\label{constant}
There is a constant $c_M > 0$ (depending only on $M$) such that for $\eta \neq \lambda_M$, 
\[
	|H_1(M(\eta))| = c_M\Delta(\eta,\lambda_M).
\]
\end{lemma}
We recall the definition of the constant $c_M$ more explicitly. Fix a basis $(\mu, \lambda_M)$ for $H_1(\partial M)$, where $\Delta(\mu,\lambda_M)=1$. Then under the homomorphism  $i_*:H_1(\partial M)\rightarrow H_1(M)\isom \Z\oplus H$, we have $i_*(\mu)=(\ell, u)$ and $i_*(\lambda_M)=(0,h)$ for some $u, h\in H$ and $\ell\in\Z$. It turns out that $|\ell| = \ord_H i_*(\lambda_M)$. The constant $c_M$ is then described as
\begin{equation}\label{eq:watson-constant}
	c_M = \ell  r_1\cdots r_k = \ord_H i_*(\lambda_M)|H|,
\end{equation}
where $r_1, \dots, r_k$ are the invariant factors of the finite abelian group $H$. In fact, Watson gives an explicit presentation for $H_1(M(\eta))$. Writing the slope $\eta=a\mu+b\lambda_M$ on $\partial M$, we see $i_*(\eta)=(a\ell, au+bh)$. A presentation matrix for $H_1(M(\eta))$ is then
\begin{equation}\label{eq:presmatrix}
	\left( \begin{array}{ccc}
	a\ell & 0 \\
	au+bh & I_{\vec{r}}
	\end{array} \right),
\end{equation} 
where $I_{\vec{r}}$ is the $k \times k$ diagonal matrix with $i$th diagonal entry given by $r_i$. In particular, $|H_1(M(\eta))| = a \ell r_1 \cdots r_k$, and $c_M$ is taken to be $\ell r_1 \cdots r_k$. Finally, note that because $\Delta(\mu,\lambda_M) = 1$, we have $a=\Delta(\eta, \lambda_M)$, and this gives the statement of Lemma~\ref{constant}.  

In the following subsection, we will use \eqref{eq:presmatrix} to study the homology class of the lift of a crossing arc in the presence of a cosmetic crossing change.

\subsection{The lift of the crossing arc}
\label{lift}
We are now prepared to describe conditions which guarantee $\gam$ represents a null-homologous class in $H_1(\bdc)$. This will allow us to apply the Dehn surgery characterization of the unknot in the case that $\bdc$ is an L-space (see Theorem~\ref{unknot}).  Let us continue with the assumptions that $K$ admits a crossing change which preserves the isotopy type of $K$, and that $\alpha$ and $\beta$ are the two filling slopes for the orientation preserving homeomorphic manifolds $\Sigma(K^+)$ and $\Sigma(K^-)$, respectively.

\begin{theorem}
\label{squarefreefactors}
Suppose that $\Sigma(K)$ is the branched double cover of a knot $K$ admitting a crossing change preserving the isotopy type of $K$, and that 
\begin{equation}
\label{decomposition}
	H_1(\Sigma(K)) \isom  \Z/d_1\Z \oplus \dots \oplus \Z/d_n\Z
\end{equation}
where each $d_i$ is square-free. Then the lift $\gam$ of the crossing arc at $c$ is trivial in $H_1(\Sigma(K))$.
\end{theorem}
\begin{proof}
Let $M = \Sigma(K) - N(\gam)$.  We consider the slopes $\alpha$ and $\beta$ on $M$ for which Dehn filling gives rise to the homeomorphic pair $M(\alpha)$ and $M(\beta)$, which are $\Sigma(K^+)$ and $\Sigma(K^-)$ respectively. Fix a basis $(\mu, \lambda_M)$ for $H_1(M)$, where $\lambda_M$ is the rational longitude and $\mu$ is any slope with $\Delta(\mu,\lambda_M)=1$. 
Write the two filling slopes $\alpha$ and $\beta$ in terms of this basis as
\begin{align*}
	\alpha &= p\mu + q\lambda_M \\
	\beta &= s\mu + r\lambda_M,
\end{align*}
with $p, s \geq 0$.  By Lemma \ref{constant}, 
\begin{eqnarray*}
	c_M\Delta(\alpha,\lambda_M)  = |H_1(M(\alpha))| = |H_1(M(\beta))| =  c_M\Delta(\beta,\lambda_M)
\end{eqnarray*}
which implies that $p=\Delta(\alpha,\lambda_M) = \Delta(\beta,\lambda_M) = s$. Thus, by the Montesinos trick, 
\[
	2 = \Delta(\alpha, \beta) = |p(q-r)|
\]
and so $p=1$ or $p=2$. However, if $p=2$, then $|H_1(\bdc)|$ is even, and this contradicts that knots have odd determinants. Therefore $p=1$, and we have by \eqref{eq:watson-constant} that 
\[
	|H_1(\bdc)| = c_M = |H|(\ord_H i_*(\lambda)).
\]
Now the distance $\Delta(\alpha, \lambda_M)$ is one, so after changing the basis $(\mu,\lambda_M)$ to $(\mu + q \lambda_M, \lambda_M)$, we may assume $\alpha = \mu$. Recall that we write $i_*(\mu) = (\ell,u) \in \mathbb{Z} \oplus H$.  Consider the presentation matrix 
\begin{equation}\label{eq:Hpresentation} 
	I_{\vec{r}} =\left( \begin{array}{cccc}
	r_1 \\
	& \ddots \\
	& & r_k
	\end{array} \right),
\end{equation}
for $H$, where $r_1, \ldots, r_k$ are the invariant factors of $H$.  From this, we may identify $u \in H$ (non-uniquely) with a vector $\vec{u} = (u_1,\ldots,u_k)$. From \eqref{eq:presmatrix}, a presentation matrix $A$ for $M(\alpha)$ is given by
\begin{equation}\label{eq:A-pres-matrix}
	A=\left( \begin{array}{ccccc}
	\ell &  \\
	u_1 & r_1 \\
	\vdots && \ddots \\
	u_k & & & r_k
	\end{array} \right),
\end{equation}
where $\ell = \pm \ord_H i_*(\lambda)$.  After possibly multiplying the first column by $-1$, we may assume $\ell \geq 1$.  

Now if $\ell=\ord_H i_*(\lambda_M)=1$, then $\lambda_M$ is integrally null-homologous and this implies that $\gam$ is in fact null-homologous in $\bdc$. To see this, consider the filling torus $N(\gam)$, and note that the rational longitude $\lambda_M$ is homologous to the core $\gam$ in $N(\gam)$, considered as a submanifold of $M(\alpha)$, since $\gam$ has intersection number one with the meridional disk bounded by $\alpha$ in $N(\gam)$. Since $\lambda_M$ is integrally null-homologous in the exterior $M$, then $\lambda_M$ also bounds in the filled manifold $M(\alpha)$. Hence $\gam$ is null-homologous in $M(\alpha)$.  Thus, it is our goal to show that $i_*(\lambda_M)$ is trivial.

As an aside, we note that in the special case each $d_i$ in (\ref{decomposition}) is a distinct prime (e.g. when $\det(K)$ is square-free) then it is immediate that $\ell=1$ because $i_*(\lambda_M)$ generates a subgroup of $H$. In general, we will argue that $\ell = 1$ by using the Smith normal form of $A$. 

Let $\Gamma_i$ denote the greatest common divisor of the determinants of the $i\times i$ minors of $A$.  Recall that the Smith normal form for an invertible $m \times m$ matrix $X$ is the diagonal matrix $I_{\vec{\delta}}$ where $\delta_i$ is $\Gamma_i/\Gamma_{i-1}$.  We have that $I_{\vec{\delta}}$ presents the same group as does $X$.  Finally, recall that $|\det(X)| = \delta_1 \cdots \delta_m$ is the order of the group being presented.   

First, we claim that each $u_i$ in \eqref{eq:A-pres-matrix} is a multiple of $\gcd(\ell, r_i)$. For if $u_i$ is not a multiple of $\gcd(\ell, r_i)$, then there exists some prime $p$ such that $p^j \divides \ell$ and $p^j \divides r_i$ but $p^j \not\divides u_i$ for some $j$. Now let $A_{1,i+1}$ be the $k \times k$ minor obtained by deleting the first row and $(i+1)$-th column of $A$; we have $|\det(A_{1,i+1})| = r_1 \cdots r_{i-1} u_i r_{i+1} \cdots r_k$. If $t$ is the largest power of $p$ which divides $\det(A)$, then the largest power of $p$ which divides $\det(A_{1,i+1})$ is at most $t-2$. Since $\delta_1 \cdots \delta_k = \Gamma_k$ divides $\det(A_{1,i+1})$, the largest power of $p$ which divides $\Gamma_k$ is at most $t - 2$.  Because $\det(A) = \delta_1 \cdots \delta_k \delta_{k+1}$ we see that $p^2$ divides $\delta_{k+1}$.  But this implies that some invariant factor in the decomposition of $H_1(\bdc)$ is not square-free, a contradiction. Thus, we may now write 
\[
	A=\left( \begin{array}{ccccc}
	\ell &  \\
	a_1\gcd(\ell, r_1) & r_1 \\
	\vdots && \ddots \\
	a_k\gcd(\ell, r_k) & & & r_k
	\end{array} \right),
\]
for some integers $a_1,\dots, a_k$.

Let us now consider the cosmetic filling. 
By Lemma~\ref{montesinostrick}, the curve $\beta$ may be written $\mu \pm 2\lambda_M$.  A presentation matrix $B$ for $M(\beta)$ is
\[ 
	B=\left( \begin{array}{ccccc}
	\ell &  \\
	u_1 \pm 2h_1 & r_1 \\
	\vdots && \ddots \\
	u_k \pm 2h_k & & &r_k
	\end{array} \right),
\] 
where $\vec{h}=(h_1,\dots, h_k)$ is identified with the image of $i_*(\lambda_M)$ in $H$ under \eqref{eq:Hpresentation}. The same argument as that for $A$ now applies to $B$ to show that $\gcd(\ell, r_i) \divides u_i \pm 2h_i$.  Because $\gcd(\ell, r_i)$ divides both $u_i$ and $|H_1(\Sigma(K))|$, and because $\gcd(\ell,r_i)$ is odd, we know $\gcd(\ell, r_i) \divides h_i$ and we write $h_i = b_i\gcd(\ell, r_i)$ for some integers $b_1, \dots, b_k$. Now, recall that $|\ell|=\ord_H i_*(\lambda_M)$. This means that $\ell \vec{h}$ is in the column span of $A_{1,1}$, so for each $i$, 
\[
	\ell h_i = \ell b_i \gcd(\ell, r_i) = c_i r_i
\]
for some integers $c_1, \dots, c_k$. But then 
\begin{eqnarray*}
	c_i\frac{r_i}{\gcd(\ell, r_i)} = \ell b_i &\Rightarrow& \frac{r_i}{\gcd(\ell, r_i)} \divides b_i \\
	&\Rightarrow&  b_i = g_i \frac{r_i}{\gcd(\ell, r_i)}, \text{ for some integer } g_i \\
	&\Rightarrow& h_i = g_i \frac{r_i}{\gcd(\ell, r_i)} \gcd(\ell, r_i).
\end{eqnarray*}
Because each $h_i$ is a multiple of $r_i$, then in fact $\vec{h}$ is in the column space of $A_{1,1}$.  However, $A_{1,1}$ is a presentation matrix for $H$, the torsion subgroup of $H_1(M)$, by \eqref{eq:Hpresentation}.  This says that $h$ represents the trivial element in $H$. Hence $i_*(\lambda_M)$ is trivial in $H_1(M)$, and we conclude that $\gam$ is null-homologous in $M(\alpha)$.
\end{proof}


\section{Proofs of main theorems}
\label{sec:mainproofs}

\subsection{Proof of Theorem \ref{main}}
Both Theorem \ref{main} and Theorem \ref{lspaceknots} depend on the Dehn surgery characterization of the unknot generalized to the case of null-homologous knots in L-spaces. The characterization of the unknot $U$ in $S^3$ is due to Kronheimer, Mr\'owka, Ozsv\'ath and Szab\'o \cite[Theorem 1.1]{KMOS}, and its generalization in L-spaces is due to Gainullin \cite[Theorem 8.2]{Gainullin}.  
\begin{theorem}
\label{unknot}
Let $K$ be a null-homologous knot in an $L$-space $Y$, and let $Y_{p/q}(K)$ denote the result of $p/q$-Dehn surgery surgery along $K$. If $Y_{p/q}(K)$ is orientation preserving homeomorphic to $Y_{p/q}(U)$, then $K$ is isotopic to $U$. 
\end{theorem}
With this, we are now prepared to prove the main theorem.

\begin{reptheorem}{main}
Let $K$ be a knot in $S^3$ with branched double cover $\Sigma(K)$ an L-space. If 
\[
	H_1(\Sigma(K)) \isom  \Z/d_1\Z \oplus \dots \oplus \Z/d_k\Z
\] 
with each $d_i$ square-free, then the cosmetic crossing conjecture holds for $K$.
\end{reptheorem}

\begin{proof}
Suppose that the branched double cover $\Sigma(K)$ is an L-space and $K$ admits a crossing change at a crossing $c$ which preserves the isotopy type of $K$. Without loss of generality, assume $c$ is positive.  We will show $c$ is nugatory. 

Let the knot $\gam$ in $\Sigma(K)$ be the lift of the crossing arc $\gamma$ at $c$. Theorem \ref{squarefreefactors} implies that in fact $\gam$ must be null-homologous in $\Sigma(K)$. Since $\gam$ is null-homologous and $|H_1(\Sigma(K^-))|$ is obviously equal to $|H_1(\Sigma(K))|$, we have that $\Sigma(K^-)$ is obtained by $\pm 1/n$-surgery on $\gam$. By assumption, $K^+$ is isotopic to $K^-$, thus $\Sigma(K^+)$ is orientation preserving homeomorphic to $\Sigma(K^-)$. By doing $\pm1/n$--framed Dehn surgery along an unknot in $\Sigma(K^+)$, we do not change the oriented homeomorphism type of $\Sigma(K^+)$, for any $n$. Thus applying Theorem \ref{unknot} we deduce that $\gam$ is isotopic to the unknot in $\bdc$. 

It remains to show that this implies the crossing $c$ is nugatory. This will follow as a special case of the equivariant Dehn's Lemma of Meeks and Yau \cite{MY}, which was instrumental in the proof of the Smith Conjecture. The case for an involution is due to Kim and Tollefson \cite{KT} and Gordon and Litherland \cite{GL}. The following is well-known to experts (see for instance \cite{Torisu}). However, we include the proof for completeness.  
\begin{proposition}
\label{nugatory}
If $\gam$ is a null-homologous unknot in $\Sigma(K)$, and the arc $\gamma\in S^3$ is the image of $\gam$ under the covering involution of $\Sigma(K)$, then the crossing associated with the arc $\gamma$ must be nugatory. 
\end{proposition}

\begin{proof}
Let $c$ be the crossing associated with $\gamma$ and let $M$ denote the exterior of $\gam$. By assumption $\gam$ is an unknot, therefore $M=D^2 \times S^1 \# \Sigma(K)$. We also will think of $M$ as $\Sigma(B, T)$, where $(B,T)$ is the tangle obtained by removing an open 3-ball neighborhood of the crossing arc at $c$. Finally, let $\tau$ denote the covering involution on $\Sigma(K)$.  

Let $\tGam  = D^2 \times \{pt\}$ in $M = D^2 \times S^1 \# \Sigma(K)$.  Of course, $\tGam$ is a compressing disk and $\partial \tGam$ is essential in $\partial M$.  Since $\partial \tGam$ is the unique slope on $\partial M$ which bounds in $M$, we see that $\partial \tGam$ is the rational longitude $\lambda_M$. By the equivariant Dehn's Lemma, we may assume that either $\tau(\tGam) \cap \tGam$ is empty or $\tau(\tGam) = \tGam$.  

First, suppose that $\tau(\tGam) \cap \tGam$ is empty.  This implies that $\tGam$ descends to a disk $\Gam$ in the exterior of $T \subset B$, and thus the exterior of $K$, since the lift of $K$ to $\Sigma(K)$ is precisely the fixed point set of $\tau$.  It remains to see that $\partial \Gam$ is a crossing circle for the crossing $c$, since this would imply that $c$ is nugatory.  We recall our previous notation $M(\alpha) = \Sigma(K^+)$ and $M(\beta) = \Sigma(K^-)$. By the same arguments as in the proof of Theorem \ref{squarefreefactors}, we have that 
$\Delta(\alpha, \beta)=2$ and $\Delta(\alpha, \lambda_M)=\Delta(\beta, \lambda_M)=1$. Taking again $(\alpha, \lambda_M)$ as a basis for $H_1(\partial M)$ we further have that $\beta = \alpha\pm 2\lambda_M$ in this basis. It is straightforward to verify that for any slope $\eta$ on $\partial M$, if $\Delta(\alpha, \eta) = \Delta(\beta,\eta) = 1$, then either $\eta = \lambda_M$ or $\eta = \pm \alpha + \lambda_M$ (where the sign is determined by $\beta = \alpha \pm 2 \lambda_M$).  In particular, there are precisely two slopes on $\partial M$ which have distance one from each of $\alpha$ and $\beta$.  

\begin{figure}
	\centering
	\includegraphics[width=5cm]{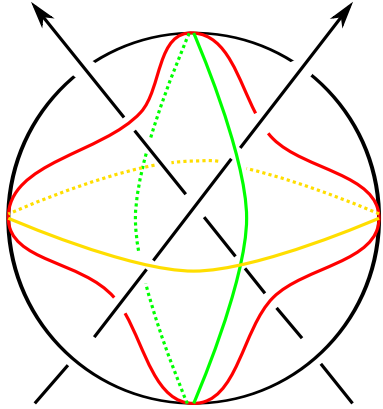}
	\caption{The two curves $w_1$ and $w_2$ in yellow and green, respectively, along with the quotient of $\alpha$ under the action of $\tau$ in red.}
	\label{tangleball}
\end{figure}
Consider the two curves $w_1$ and $w_2$ on the boundary of the neighborhood of the crossing $c$ shown in Figure~\ref{tangleball}. These curves lift to distinct slopes on $\partial M$ and have distance one from each of $\alpha$ and $\beta$. It follows that the lift of either $w_1$ or $w_2$ must be $\lambda_M = \partial \tGam$.  In other words, we must have that $\tGam$ descends to a disk in the exterior of $K$ with either $w_1$ or $w_2$ as its boundary. As the crossing circle (i.e. $w_2$ in Figure \ref{tangleball}) is the curve which is null-homologous in the complement of $K$, this implies that the crossing circle bounds a disk in the exterior of $K$ and thus $c$ is nugatory, as desired.         

Now, we consider the case that $\tau(\tGam) = \tGam$. Noting that the fixed point set of $\tau$ restricted to $\tGam$ is at most one-dimensional, Smith theory implies that $\tGam$ intersects $\fix(\tau)$ in either an arc or a single point. 
In the case this intersection is an arc, standard arguments allow us to perturb $\tGam$ such that $\tau(\tGam) \cap \tGam$ is empty, which was handled in the previous case.  Thus, it remains to consider the case that $\tGam$ intersects the fixed point set of $\tau$ in a single point.  In this case, it follows that $\tGam$ descends to a disk $\Gam$ in $B$ which intersects $T = \fix(\tau|_M)$ in a single point. Note that $\partial \Gam$ also bounds a disk $\Gam'$ contained in $\partial B$ which intersects $\partial T$ in one point. But now $\Gam'$ lifts to a disk $\tGam'$ in $\Sigma(B,T) = M$ which is contained in $\partial M$ and has $\partial \tGam' = \partial \tGam$. This contradicts that $\tGam$ is a compressing disk, since $\partial \tGam$ bounds a disk in $\partial M$.  
\end{proof}

The above proposition completes the proof of Theorem~\ref{main}.  
\end{proof}

\begin{remark}
While we are not able to prove the cosmetic crossing conjecture in the case that $H_1(\Sigma(K))$ has summands which are not square-free, Theorem~\ref{main} can still be useful in the following sense.  If a knot $K$ has $\Sigma(K)$ an L-space, then if a particular crossing arc lifts to a null-homologous knot, then that crossing change must either be nugatory or change the isotopy class of the knot. 
\end{remark}

\begin{remark}
\label{generalization}
Although our main focus is on the cosmetic crossing conjecture for knots in the 3-sphere, the proof of Theorem \ref{main} also applies to knots in arbitrary integer homology spheres. For a knot $K$ in an integer homology sphere $Y$, we need only modify our definition of a crossing change; rather than change a crossing from positive to negative in a diagram of $K$, we define a {\em crossing change} as a $\pm1$--framed Dehn surgery along the crossing circle $C$, where $C$ is the boundary of a crossing disk $D$ as above. A \emph{cosmetic crossing change} for $K$ in $Y$ is analogously defined as for knots in $S^3$.

\begin{theorem}
\label{generaltheorem}
Let $K$ be a knot in an integer homology sphere $Y$ such that the double cover of $Y$ branched over $K$ is an L-space. If each integer $d_i$ in the decomposition
\[
	H_1(\Sigma(Y,K)) \isom \Z/d_1\Z \oplus\dots \oplus \Z/d_k \Z
\]
is square-free, then $K$ admits no cosmetic crossing changes. 
\end{theorem}
\begin{proof}
Note that Proposition \ref{nugatory} applies to any closed, oriented 3-manifold with a $\Z/2\Z$-action, and the Montesinos trick may be applied in any integer homology sphere. The content of Section \ref{rationallongitude} regarding the rational longitude is valid for any compact, oriented 3-manifold $M$ with torus boundary and $H_1(M;\Q)\isom \Q$, and Theorem \ref{squarefreefactors} will hold whenever $\Sigma(Y,K)$ is the double cover of an integer homology sphere $Y$ branched over a knot $K$, satisfying the same homology condition as (\ref{decomposition}). Thus the proof of Theorem \ref{main} applies mutatis mutandis.
\end{proof}

Boyer, Gordon and Watson conjectured that an irreducible rational homology 3-sphere is an L-space if and only if its fundamental group is not left-orderable \cite[Conjecture 1]{BGW}. An affirmative answer to this conjecture, together with a result of Boyer, Rolfsen and Wiest \cite[Theorem 3.7]{BRW}, suggest the following.
\begin{conjecture}
\label{lspace}
Let $Y$ be a rational homology 3-sphere. If $\Sigma(Y,K)$ is a prime L-space, then $Y$ is an L-space.
\end{conjecture}
Another conjecture, due to \os (see for instance \cite{HL:Splicing}), is that the only integer homology sphere L-spaces are the 3-sphere and connected sums of the Poincar\'e sphere. We remark that these conjectures conspire to limit the scope of Theorem \ref{generaltheorem} to the 3-sphere and connected sums of the Poincar\'e sphere.
\end{remark}

\subsection{Proof of Theorem \ref{lspaceknots}}
Next we turn our attention from knots whose branched double covers are L-spaces to knots which admit L-space surgeries. Recall that an L-space knot is a knot in $S^3$ such that a positive $p/q$--Dehn surgery along $K$ yields an L-space $S^3_{p/q}(K)$. By work of Ni, an L-space knot $K$ is fibered \cite{Ni}, thus Kalfagianni's result for fibered knots immediately implies that all L-space knots satisfy the cosmetic crossing conjecture. Thus the following theorem was previously known to be true, but the proof provided here does not appeal to Kalfagianni's work.

\begin{reptheorem}{lspaceknots}
Let $K$ be a knot in $S^3$ with an L-space surgery. Then $K$ satisfies the cosmetic crossing conjecture.
\end{reptheorem}

\begin{proof}
After perhaps mirroring $K$, we can assume that $K$ admits a positive L-space surgery.  Suppose that $K$ admits a crossing change at $c$ that preserves the knot type of $K$. Without loss of generality, assume that $c$ is a positive crossing; the argument in the case $c$ is negative is identical. Recall that if $K$ is an L-space knot, then $p$--framed surgery along $K$ in $S^3$ yields an L-space for all integers $p\geq 2g(K)-1$. Denote by $\widetilde{C}_p$ the image of the crossing circle $C$ after $p$--surgery along $K$.  Note that $\widetilde{C}_p$ is null-homologous.  It follows that
\[
	(S^3_p(K^+))_{-1}(\widetilde{C}_p)\isom S^3_p(K^-)
\]
where $K^-$ is the image of $K^+$ after $-1$--surgery along $C$. By assumption, the pairs $(S^3, K^+)$ and $(S^3, K^{-})$ are orientation preserving homeomorphic. Thus
\[
	(S^3_p(K^+))_{-1}(\widetilde{C}_p) \isom S^3_p(K^-) \isom S^3_p(K^+) \isom (S^3_p(K^+))_{-1}(U),
\]
where $U$ is an unknot in $S^3_p(K^+)$. As another application of Theorem \ref{unknot}, we find that $\widetilde{C}_p$ is itself an unknot for each $p \geq 2g(K) - 1$. In particular, this implies that the manifold with boundary $S^3_p(K)-N(\widetilde{C}_p)$ is reducible for all $p\geq 2g(K)-1$.

It is well known that if $M$ is a compact, orientable, irreducible 3-manifold and $F\subset \partial M$ is a toroidal boundary component, then at most finitely many slopes on $F$ yield reducible Dehn fillings. The exterior of the two-component link $K\cup C$ in $S^3$ has infinitely many fillings along $K$ yielding the reducible manifolds $S^3_p(K)-N(\widetilde{C}_p)$, so it must be the case that the exterior of $(K\cup C)$ is reducible.  It follows that $K \cup C$ is split.  Since $C$ sits in an embedded 3-ball,  the crossing is nugatory.
\end{proof}


\section{Examples}
\label{sec:examples}

\subsection{Knots of at most ten crossings}
\label{smallknots}
In this section we settle the cosmetic crossing conjecture for all but ten of the knots in the Rolfsen tables. Let us summarize the known obstructions to a knot $K$ admitting a cosmetic crossing change. If:
\begin{itemize}
\item $K$ is two-bridge \cite{Torisu}, or
\item $K$ is fibered \cite{Kalfagianni}, or
\item $K$ is genus one and is not algebraically slice or $H_1(\bdc)$ fails to be cyclic \cite{BFKP}, or
\item $\Sigma(K)$ is an L-space and $H_1(\Sigma(K))$ has only square-free summands [Theorem \ref{main}],
\end{itemize}
then $K$ admits no cosmetic crossing changes. With the exception of the knots $9_{46}$ and $10_{128}$, all of the knots mentioned in Table \ref{nine} and Table \ref{ten} are known to be alternating or quasi-alternating, hence are $\rkh$--thin and therefore have branched double covers that are L-spaces. The quasi-alternating status is the collected effort of many authors and relevant summaries may be found in \cite{CK, Jablan}.

\begin{table}[h]
\[ 
\begin{array}{cccc}
	\text{Knot} & \text{Determinant} & \text{Genus} & H_1(\Sigma(K)) \\
	\hline
	\hline
	8_{15} & 33 & 2 &\\
	9_{16} & 39 & 3 &\\
	9_{25} & 47 & 2 &\\
	9_{35} & 27 & 1 & \Z/3\Z \oplus \Z/9\Z\\
	9_{37} & 45 & 2 & \Z/3\Z \oplus \Z/15\Z  \\
	9_{38} & 57 & 2 &\\
	9_{39} & 55 & 2 &\\
	9_{41} & 49 & 2 & \Z/7\Z \oplus \Z/7\Z \\
	9_{46} & 9 & 1 & \Z/3\Z \oplus \Z/3\Z \\
	9_{49}  & 25 & 2 & \Z/5\Z \oplus \Z/5\Z  \\
	\hline
	\hline
\end{array}
\]

\caption{Knots of nine or fewer crossings that are non-fibered and have bridge number at least three.  When the determinant is not square-free, we provide the first homology of the branched double cover.}
\label{nine}
\end{table}
\begin{proposition}
\label{ninecrossing}
All knots of at most nine crossings satisfy the cosmetic crossing conjecture.
\end{proposition}
\begin{proof}
Utilizing the KnotInfo database \cite{KnotInfo}, in Table \ref{nine} we collect the knots of nine or fewer crossings that are non-fibered and have bridge number at least three. These are listed with their determinants and genera. Where indicated, the homology groups of the branched double covers were computed using the ``Cyclic Branched Cover Homology Calculator" program, available at KnotInfo.  

The knots $8_{15}, 9_{16}, 9_{25}, 9_{37}, 9_{38}, 9_{39}, 9_{41}$, and $9_{49}$ are quasi-alternating, hence $\rkh$--thin.  By the computations of the first homologies of the branched double covers in Table \ref{nine}, Theorem \ref{main} implies they admit no cosmetic crossing changes. The knot $9_{46}$ is known to be reduced Khovanov thin, but not \emph{odd} Khovanov homology thin, due to Shumakovitch \cite{Shumakovitch}. Thus it is not quasi-alternating, but its branched double cover is still an L-space, and Theorem \ref{main} again applies. Theorem \ref{main} does not apply to the only remaining knot, $9_{35}$, because $H_1(\Sigma(9_{35}))$ has a $\Z/9\Z$ summand. Since this knot has genus one and $H_1(\bdc)$ is not cyclic, it admits no cosmetic crossing changes by \cite{BFKP}.  
\end{proof}

\begin{proof}[Proof of Theorem \ref{classification}]
Using the same strategy, we collect the relevant data for the ten-crossings knots that are non-fibered and have bridge number at least three. All of the knots in Table \ref{ten} have genus at least two, so the obstruction of \cite{BFKP} will not help in the present case.  Aside from the knot $10_{128}$, each knot in Table \ref{ten} is $\rkh$--thin, so Theorem \ref{main} applies to each knot other than those listed in \eqref{exceptions}.  Finally, we consider the knot $10_{128}$.  This is the Montesinos knot $M(-2; 4/7, 1/2, 2/3)$. Using \cite[Theorem 1.3]{LiscaMatic} and \cite[Theorem 1.1]{LiscaStipsicz}, the branched double cover may be checked to be an L-space, and again Theorem \ref{main} applies.  This analysis of the ten-crossing knots together with Proposition \ref{ninecrossing} completes the proof of Theorem \ref{classification}.
\end{proof}

\begin{remark}
The ten knots in \eqref{exceptions} share the following additional properties: they are hyperbolic, non-fibered, bridge number three, and have genus either two or three. 
\end{remark}

\begin{table}
\centering
\begin{tabular}{l|l}
$
\begin{array}{cccc}
\text{Knot} & \text{Determinant} & H_1(\Sigma(K)) \\
\hline
\hline
10_{49}	& 59	 & \\
10_{50}	& 53	 & \\
10_{51}	& 67	 & \\
10_{52}	& 59	 & \\
10_{53}	& 73	 & \\
10_{54}	& 47	 & \\
10_{55}	& 61	 & \\
10_{56}	& 65	 & \\
10_{57}	& 79	 & \\
10_{58}	& 65	 & \\
10_{61}	& 33	 & \\
10_{63}	& 57	 & \\
\rowcolor{Gray}
10_{65}	& 63 & \Z/63\Z \\
\rowcolor{Gray}
10_{66}	& 75 & \Z/75\Z \\
\rowcolor{Gray}
10_{67}	& 63& \Z/63\Z \\
10_{68}	& 57	 & \\
10_{72}	& 73	 & \\
10_{74}	& 63 & \Z/3\Z \oplus \Z/21\Z \\
10_{76}	& 57	 & \\
\rowcolor{Gray}
10_{77}	& 63 & \Z/63\Z \\
10_{80}	& 71	 & \\
10_{83}	& 83	 & \\
10_{84}	& 87	 & \\
10_{86}	& 85	 & \\
\rowcolor{Gray}
10_{87}	& 81 & \Z/81\Z \\
10_{90}	& 77	 & \\
10_{92}	& 89	 & \\
10_{93}	& 67	 & \\
\hline
\end{array}
$
&
$
\begin{array}{ccc}
\text{Knot} & \text{Determinant} & H_1(\Sigma(K)) \\
\hline
\hline
10_{95}	& 91	 & \\
10_{97}	& 87	 & \\
\rowcolor{Gray}
10_{98}	& 81 & \Z/3\Z \oplus \Z/27\Z \\
10_{101}	& 85	 & \\
10_{102}	& 73	 & \\
10_{103}	& 75 & \Z/5\Z \oplus \Z/15\Z \\
\rowcolor{Gray}
10_{108}	& 63 & \Z/63\Z \\
10_{111}	& 77	 & \\
10_{113}	& 111 & \\
10_{114}	& 93	 & \\
10_{117}	& 103 & \\
10_{119}	& 101 & \\
10_{120}	& 105 & \\
10_{121}	& 115 & \\
10_{122}	& 105 & \\
10_{128} & 11 & \\
\rowcolor{Gray}
10_{129}	& 25	 & \Z/25\Z\\
10_{130}	& 17 & \\
10_{131}	& 31 & \\
10_{134}	& 23 & \\
10_{135}	& 37 & \\
10_{142}	& 15 & \\
10_{144}	& 39 & \\
10_{146}	& 33	 & \\
\rowcolor{Gray}
10_{147}	& 27 & \Z/27\Z \\
10_{162}	& 35	 & \\
\rowcolor{Gray}
10_{164}	& 45 & \Z/45\Z \\
10_{165}	& 39 & \\	
\hline
\end{array}
$
\end{tabular}
\caption{The ten-crossing knots that are non-fibered and have bridge number at least three.  When the determinant is not square-free we provide the first homology of the branched double cover, to determine if Theorem \ref{main} applies.}
\label{ten}
\end{table}

\subsection{Further examples}
Using the obstructions of Theorem \ref{main}, it is not difficult to find new families of knots which satisfy the cosmetic crossing conjecture. 

\begin{example}[Pretzel knots]
\label{pretzelexample}
Consider the three-stranded pretzel knots $P(-p,q,r)$ where $p>0$ is even and $q,r>0$ are odd. When $q=p-1$, we have that 
\begin{equation}
\label{det}
	\det(P(-p,q, r)) = | -pq-pr+qr| = p^2-p +r,
\end{equation}
thus for any odd integer $n$, the pretzel knots $P(-p, p-1, n+p-p^2)$ are an infinite family of knots of determinant $n$. By Greene \cite[Theorem 1.4(d)]{Greene}, the pretzel knots $P(-p,p-1,r)$ are quasi-alternating, but in particular $r$ must be positive. Combining these statements, we see that for every odd integer $n \geq 3$ we have a finite set of quasi-alternating pretzel knots $P(-p,p-1,r)$ of determinant $n$, where $r=n+p-p^2$. By Corollary \ref{squarefreedet}, we see that for every square-free odd integer $n \geq 3$, there exists even $p > 0$ such that the pretzel knot $P(-p,p- 1,n+p-p^2)$ has determinant $n$ and satisfies the cosmetic crossing conjecture.

Note also that when $p\geq 4$ is even and $q, r \geq3$ are odd, then by \cite{Gabai:Detecting}, $P(-p,q,r)$ is non-fibered, and the genus of $P(-p,q,r)$ is $(q+r)/2$ (see \cite[Corollary 2.7]{KL}). Moreover, these knots are hyperbolic \cite[Theorem 2.4]{KL} and not two-bridge, and so excluding the handful of cases where $0<p,q,r<4$, the previously known obstructions to admitting a cosmetic crossing change do not apply. 
\end{example}

\begin{example}[Branch sets of L-space surgeries]
Let $K$ be a strongly invertible L-space knot.  For $p/q \geq 2g(K) - 1$ with $p$ odd, we have that $S^3_{p/q}(K)$ is an L-space which can be expressed as the branched double cover of a knot $J_{p/q}$.  If $p$ is square-free, then we have by Theorem~\ref{main} that $J_{p/q}$ necessarily satisfies the cosmetic crossing conjecture.  Further, if $K$ is hyperbolic, then by Thurston's hyperbolic Dehn surgery theorem, all but finitely many of the surgeries on $K$ will be hyperbolic. For such surgeries, the corresponding knot $J_{p/q}$ is not arborescent (and in particular is not a pretzel or Montesinos knot).  While $K$ is necessarily fibered, the quotient knot $J_{p/q}$ need not be fibered; for example, every 2-bridge knot is the quotient of a surgery on the unknot.

Finally, we remark that conjecturally every L-space knot is strongly invertible.
\end{example}

In Example \ref{pretzelexample}, we made use of finite sets of quasi-alternating knots of a fixed determinant. In \cite[Conjecture 3.1]{Greene}, Greene conjectured that this phenomenon is always the case, namely that there exist only finitely many quasi-alternating links with a given determinant. We now describe an infinite family of knots with fixed determinant. Though these knots will be $\rkh$--thin, presumably they are not quasi-alternating.

\begin{example}[Symmetric unions]
\label{symmetricexample}
Further examples can be generated using symmetric unions, a classical construction due to Kinoshita and Terasaka \cite{KiTe}. In particular this construction can be used to create infinite families of examples with a fixed determinant, unlike the two examples described above. The symmetric unions $K_n(5_2)$ of the knot $5_2$ (see Figure \ref{symmetricunion} for the knot $K_2(5_2)$) are an example of such a construction.  

In a forthcoming note of the second author \cite{Moore}, the knots $K_{n} = K_n(5_2)$ with $n\equiv 0 \pmod 7$, are shown to be reduced Khovanov thin, non-alternating, non-fibered, hyperbolic, of genus two, bridge number three, and have $H_1(\Sigma(K_n))\isom \Z/7\Z \oplus \Z/7\Z$. Applying Theorem \ref{main}, we see that this infinite family of knots with constant determinant satisfies the cosmetic crossing conjecture.  
\end{example}

\begin{figure}[t]
\centering
	\labellist
	\huge 
	\pinlabel $\rbrace$ at 400 150
	\large
	\pinlabel $n$ at 450 150
	\endlabellist
	\includegraphics[width=6cm]{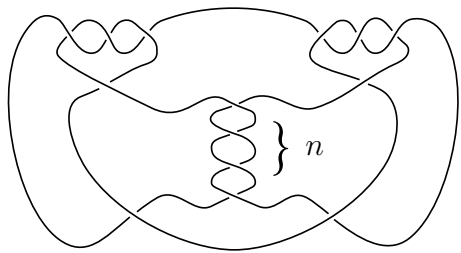}
\caption{A symmetric union of the knot $5_2$ with $n=2$ full twists in the clasp.}
\label{symmetricunion}
\end{figure}

\section*{Acknowledgments}
We would like to thank Cameron Gordon for explaining the proof of Proposition~\ref{nugatory}, Danny Ruberman and Steven Sivek for insightful conversations, and Effie Kalfagianni, Mark Powell, and Liam Watson for helpful correspondence. The first author was partially supported by NSF RTG grant DMS-1148490.  The second author is partially supported by NSF grant DMS-1148609.

\bibliographystyle{alpha}
\bibliography{bibliography}

\end{document}